\documentclass{mwart}
\usepackage[utf8]{inputenc}
\usepackage[all]{xy}

\usepackage{makeidx}

\usepackage{amsmath,amssymb,amsthm}

\usepackage{authblk}

\usepackage{amsthm}
\newtheorem{theorem}{Theorem}[section]
\newtheorem{proposition}[theorem]{Proposition}
\newtheorem{lemma}[theorem]{Lemma}

\newtheorem{definition}[theorem]{Definition}
\newtheorem{remark}[theorem]{Remark}

\newtheorem{dovermann-suh}[theorem]{Dovermann--Suh Conjecture}

\newtheorem{smith theory}[theorem]{Smith Theory Conclusions}

\def\mbb#1{\mathbb{#1}} 
\def\msl#1{\text{\sl #1}} 
\def\msf#1{\text{\sf #1}}

\usepackage{xparse}

\NewDocumentCommand{\I}{o}{
\IfNoValueTF{#1}{%
\ensuremath{\mathbb{I}}
}{%
\ensuremath{\mathbb{I}^{#1}}}}

\NewDocumentCommand{\R}{o}{
\IfNoValueTF{#1}{%
\ensuremath{\mathbb{R}}
}{%
\ensuremath{\mathbb{R}^{#1}}}}

\usepackage{tikz, tikz-cd}
\usetikzlibrary{matrix,decorations.pathmorphing,arrows}
\tikzcdset{
arrow style=tikz,
diagrams={>={Stealth[scale=0.9]}}
}

\usepackage[toc,page]{appendix}

\NewDocumentCommand{\C}{o}{
\IfNoValueTF{#1}{%
\ensuremath{\mathbb{C}}
}{%
\ensuremath{\mathbb{C}^{#1}}}}

\NewDocumentCommand{\Z}{o}{
\IfNoValueTF{#1}{%
\ensuremath{\mathbb{Z}}
}{%
\ensuremath{\mathbb{Z}^{#1}}}}

\title{Fixed point sets of smooth \(G\)-manifolds\\ pseudo-equivalent to a \(G\)-template}

\author{Krzysztof M. Pawa{\l}owski and Jan Pulikowski}

\makeindex

\usepackage{epigraph}



\makeatletter
\@addtoreset{equation}{section}
\makeatother

\usepackage{url}

\begin{document}
\maketitle

\abstract{For a finite group \(G\) not of prime power order, Oliver \cite{Oliver 1996} has answered the question
which manifolds occur as the fixed point sets of smooth actions of \(G\) on disks (resp., Euclidean spaces). 
We extend Oliver's result to compact (resp., open) smooth \(G\)-manifolds \(M\) pseudo-equivalent to \(Y\), a finite 
\(\mbb{Z}\)-acyclic \(G\)-CW complex  such that the fixed point set \(Y^G\)  is non-empty, connected, and 
\(\chi(Y^G) \equiv 1 \pmod{n_G}\), where \(n_G\) is the Oliver number of \(G\). We prove that the answer 
to the question above does not depend on the choice of \(Y\).

For a finite connected \(G\)-CW complex \(Y\) such that \(Y^G\) is non-empty and connected, called a \emph{\(G\)-template}, 
we prove that a compact stably parallelizable manifold \(F\) occurs as the fixed point set \(M^G\) of a compact smooth \(G\)-manifold \(M\)
pseudo-equivalent to \(Y\), if and only if \(\chi(F) \equiv \chi(Y^G) \pmod{n_G}\). Moreover, there exists a compact 
smooth fixed point free \(G\)-manifold pseudo-equivalent to a \(G\)-template \(Y\), if and only if \(\chi(Y^G) \equiv 0 \pmod{n_G}\). 
In particular, similarly as for actions on disks, there exists a compact smooth fixed point free \(G\)-manifold
pseudo-equivalent to the real projective space \(\mbb{R}{\rm P}^{2n}\)\! for an integer \(n \geq 1\), if and only if 
\(G\) is an Oliver group. Finally, we prove that each finite Oliver group \(G\) has a smooth fixed point free action on 
\(\mbb{R}{\rm P}^{2n}\)\! itself for some integer \(n \geq 1\).}

\medskip
\noindent 
\emph{\msl{2020} Mathematics Subject Classification}. Primary: 57S17, 57S25. 

\smallskip
\noindent
\emph{Key words}. Finite group,  smooth action, fixed point set,  fixed point free action, \(G\)-CW complex, \(G\)-vector bundle, \(G\)-template,
pseudo-equivalence.

\section{Introduction}
\label{sec: introduction}

In the theory of transformation groups, a fundamental role plays the study of smooth actions of compact Lie groups on smooth manifolds.
It is well-known that if a compact Lie group \(G\) acts smoothly on a smooth manifold \(M\), the set of points in \(M\) fixed under the action of \(G\),
\[
M^G = \{x \in M \colon \ g x = x  \ \text{for each} \ g \in G\},
\]
is a smooth submanifold of \(M\) (see, e.g., \cite[Corollary 2.5, p. 309]{Bredon 1972}. Moreover, if  \(M\) is compact, so is the manifold \(F = M^G\), 
and if the boundary \(\partial M\) of \(M\) is empty, \(\partial F\) is empty too. More generally, \(F \cap \partial M = \partial F\) by the Slice Theorem 
(see, e.g., \cite[Corollary 2.4, p. 308]{Bredon 1972}).

Another restrictions on manifolds \(F\) occurring as the fixed points sets arise from the Smith Theory (see \cite[Chapter III]{Bredon 1972}). 
A particular result asserts that if a finite \(p\)-group \(G\) (for any prime \(p\)) acts smoothly on some \(\mbb{F}_p\)-acyclic smooth manifold \(M\), 
then the manifold \(F = M^G\) is also \(\mbb{F}_p\)-acyclic.\footnote{Henceforth, \(\mbb{F}_p\) denotes the field consisting of \(p\) elements.}

A breakthrough step toward obtaining the converse statements to the Smith Theory goes back to 
Jones \cite{Jones 1971}.  For \(G = C_p\), the cyclic group of order \(p\) (where \(p\) is any prime), and a finite \(\mbb{F}_p\)-acyclic CW complex \(F\), 
Jones \cite[Theorem~1.1]{Jones 1971} has constructed a finite contractible \(G\)-CW complex \(X\) such that \(X^G = F\). 
Also, assuming that \(F\) is a compact \(\mbb{F}_p\)-acyclic stably complex submanifold of a disk \(D\) with 
a closed tubular neighborhood \(N\) of \(F\) in \(D\), upon which \(G\) acts smoothly so that \(N^G = F\), 
Jones \cite[Section~3]{Jones 1971} has described a method (which relies on \(G\)-equivariant handle addition) for extending of 
the action of \(G\) on \(N\) to a smooth action of \(G\) on \(D\) such that \(D^G = F\).

Let \(G\) be a finite \(p\)-group. The Smith Theory and the results of Jones \cite{Jones 1971} combined with induction arguments 
show that a compact smooth manifold \(F\) is the fixed point set of a smooth action of \(G\) on a disk, if and anly if \(F\) is both
\(\mbb{F}_p\)-acyclic and stably complex.\footnote{Each \(\mbb{F}_2\)-acyclic manifold is stably complex (see \cite[the proof of (3.2)]{Edmonds-Lee 1975}).}
Similarly, a smooth manifold \(F\) with \(\partial F = \varnothing\) is the fixed point set of a smooth action of \(G\) on some Euclidean space,  
if and only if \(F\) is both \(\mbb{F}_p\)-acyclic and stably complex (see \cite[Theorem A]{Pawalowski 1989}). These results are generalized 
in the article \cite{Pawalowski-Pulikowski 2019}, and in Theorem~\ref{thm: Y p-acyclic} below, we recall the conclusions. 
First, we introduce some definitions.

For a finite group \(G\), a \(G\)-map \(f \colon X \to Y\) between two \(G\)-CW complexes \(X\)~and \(Y\) is called a \emph{pseudo-equivalence}
if \(f\) is a homotopy equivalence, but not necessarily a \(G\)-homotopy equivalence. If \(G\) acts trivially on \(Y\), such \(f\) factors through 
the orbit space \(X/G\).  If \(X\) \emph{is pseudo-equivalent to} \(Y\), i.e., there exists a pseudo-equivalence  \(f \colon X \to Y\), we write \(X \simeq_{/G} Y\).

A finite connected \(G\)-CW complex \(Y\) is said to be a \emph{\(G\)-template} if the fixed point set \(Y^G\) is non-empty and connected.
In the special case where \(Y^G = Y\) (i.e., \(G\) acts trivially on \(Y\)), \(Y\) is said to be a \emph{trivial \(G\)-template}.

In this paper, our goal is to answer the question which smooth manifolds \(F\) are diffeomorphic to the fixed point sets \(M^G\) of smooth \(G\)-manifolds  
\(M \simeq_{/G} Y\) for a \(G\)-template \(Y\), and we reach the goal for specific \(G\)-templates \(Y\).

In \cite{Pawalowski-Pulikowski 2019}, we have answered the question above for any finite \(p\)-group \(G\) and any \(\mbb{F}_p\)-acyclic \(G\)-template \(Y\) 
such that \(Y^G = \msl{pt}\). The answer is restated below (cf., \cite[Theorem~A]{Pawalowski-Pulikowski 2019}). If \(Y = \msl{pt}\), 
the result goes back to \cite{Jones 1971} (resp., \cite{Pawalowski 1989}), where \(M\) can be chosen to be a disk (resp., Euclidean space).

\begin{theorem}\label{thm: Y p-acyclic}
Let \(G\) be a finite \(p\)-group and let \(Y\) be an \(\mbb{F}_p\)-acyclic \(G\)-template such that \(Y^G = \msl{pt}\).
Then a smooth manifold \(F\) is diffeomorphic to the fixed point set \(M^G\) of a compact $($resp., open$)$ stably complex \(G\)-manifold \(M \simeq_{/G} Y\), 
if and only if $F$ is compact, \(\mbb{F}_p\)-acyclic, and stably complex  $($resp.,  \(\partial F = \varnothing\) and \(F\) is \(\mbb{F}_p\)-acyclic and stably complex$)$.
\end{theorem}

Now, for a finite group \(G\) not of prime power order, and a \(G\)-template \(Y\), we  focus on answering the question 
which smooth manifolds \(F\) are the fixed point sets \(M^G\) of smooth \(G\)-manifolds \(M \simeq_{/G}Y\).

Oliver \cite{Oliver 1975} proved that once \(G\) is a finite group not of prime power order, there exists an integer \(n_G \geq 0\) such that 
a finite CW complex \(F\) is the fixed point set \(X^G\) of a finite contractible \(G\)-CW complex \(X\), if and only if the Euler  characteristic 
\(\chi(F)\) satisfies the condition that \(\chi(F) \equiv 1 \pmod{n_G}\).

The integer \(n_G\) is defined by showing that the set
\[
\{\chi(X^G) - 1 \colon \text{for all finite contractible \(G\)-CW complexes \(X\)}\}
\]
is a subgroup of \(\mbb{Z}\), the group of integers, and thus, it has the form \(n_G \cdot \mbb{Z}\) for a unique integer \(n_G \geq 0\).
We refer to \(n_G\) as to the \emph{Oliver number} of \(G\).

For a finite group \(G\) not of prime power order, and any finite dimensional, countable CW complex \(F\), Oliver \cite{Oliver 1996} 
has determined (implicitly) a subgroup of \(\widetilde{KO}(F)\), the reduced real \(K\)-theory of \(F\), which depends on some 
algebraic properties of the group \(G\).  In Section~\ref{sec: tangent bundle}, we define this subgroup of \(\widetilde{KO}(F)\) and henceforth, 
we denote it by \(\widetilde{KO}(F,G)\).

First, we answer the question which smooth manifolds \(F\) are the fixed point sets \(M^G\) of smooth \(G\)-manifolds \(M \simeq_{/G} Y\) for 
a \(\mbb{Z}\)-acyclic \(G\)-template \(Y\) such that \(\chi(Y^G) \equiv 1 \pmod{n_G}\); e.g., such that \(Y^G = \msl{pt}\) or \(Y\) is trivial.
For \(Y = \msl{pt}\), smooth \(G\)-manifolds \(M \simeq_{/G} Y\) are contractible and the result below goes back to Oliver \cite{Oliver 1996}, 
where \(M\) can be chosen to be a disk (resp., Euclidean space).

\begin{theorem}\label{thm: Y acyclic}
Let \(G\) be a finite group not of prime power order, and let~\(Y\) be a \(\mbb{Z}\)-acyclic \(G\)-template such that \(\chi(Y^G) \equiv 1 \pmod{n_G}\). 
Then a smooth manifold \(F\) is diffeomorphic to the fixed point set \(M^G\)\! of a compact $($resp., open$)$ 
smooth \(G\)-manifold \(M \simeq_{/G} Y\), if and only if \(F\) is compact, \(\chi(F) \equiv \chi(Y^G) \pmod{n_G}\), and \([TF] \in \widetilde{KO}(F,G)\)
$($resp., \(\partial F = \varnothing\) and \([TF] \in \widetilde{KO}(F,G)\)$)$.
\end{theorem}

Now, for a finite group \(G\) not of prime power order, and a \(G\)-template \(Y\), we answer the question which stably parallelizable manifolds \(F\)
are the fixed point sets \(M^G\) of compact smooth \(G\)-manifolds \(M \simeq_{/G} Y\).

\begin{theorem}\label{thm: F stably parallelizable}
Let \(G\) be a finite group not of prime power order, and let~\(Y\) be a \(G\)-template. Then a stably parallelizable manifold \(F\) is diffeomorphic to 
the fixed point set \(M^G\) of a compact smooth \(G\)-manifold \(M \simeq_{/G} Y\), if and only if \(F\) is compact and \(\chi(F) \equiv \chi(Y^G) \pmod{n_G}\).
\end{theorem}

Next, we show that for a \(G\)-template \(Y\), the Euler characteristic \(\chi(Y^G)\) is the only obstruction for the existence of a compact smooth 
fixed point free \(G\)-manifold \(M \simeq_{/G} Y\). In particular,  if the \(G\)-template \(Y\)  is trivial,
\(\chi(Y)\) is the only obstruction for the existence of such a \(G\)-manifold \(M\).

\begin{theorem}\label{thm: F empty}
Let \(G\) be a finite group not of prime power order, and let \(Y\) be a \(G\)-template.  
Then there exists a compact smooth fixed point free \(G\)-manifold \(M \simeq_{/G} Y\),  if and only if \(\chi(Y^G) \equiv 0 \pmod{n_G}\).
\end{theorem}

In the theory of transformation groups, an interesting task is to answer the question which finite groups \(G\) have smooth fixed point free 
actions on specific compact smooth manifolds which have the fixed point property (i.e., manifolds whose self-maps all have fixed points).

In the case of \(G\)-actions on disks, Oliver \cite{Oliver 1975} has answered the question as follows. 
A finite group \(G\) has a smooth fixed point free action on a disk if and only if \(G\) does not contain a series of subgroups of the form 
\(P \trianglelefteq H \trianglelefteq G\), where \(P\) and \(G/H\) are of prime power orders, and \(H/P\) is cyclic. 
A finite group \(G\) does not contain such a series of subgroups if and only if \(G\) is not of prime power order and the Oliver number 
\(n_G = 1\). In the literature, such a group \(G\)  is called an \emph{Oliver group}. This notion was used for the first time 
in the article of Laitinen and Morimoto \cite{Laitinen-Morimoto 1998}.

In general, if  a finite group \(G\) contains a subgroup which is an Oliver group, then \(G\) itself is an Oliver group.
Also, if a quotient of \(G\) is an Oliver group, so is \(G\). For \(n \geq 2\), let \(C_n\) denote the cyclic group of order \(n\), 
and let \(A_n\) and \(S_n\) denote the alternating and symmetric group on \(n\) letters, respectively.
The~list of finite Oliver groups includes finite groups such as

\begin{itemize}
\smallskip
\item[\rm 1)] nilpotent (in particular, abelian) groups with at least three noncyclic Sylow subgroups,  
                     (e.g., \(C_{pqr} \times C_{pqr}\) for three distinct primes \(p\), \(q\), and \(r\)),
\item[\rm 2)] the groups \(C_3 \times S_4\), \(S_3 \times A_4\), and \((C_3 \times A_4) \rtimes C_2\) of order \(72\),
                      which are examples of non-nilpotent solvable groups, 
\item[\rm 3)] all nonsolvable groups (e.g., most of \(\msf{GL}_n(\mathbb{F}_q)\)'s) and thus,
                     all perfect groups (e.g., most of \(\msf{SL}_n(\mathbb{F}_q)\)'s) and all nonabelian simple groups 
                    (e.g., most of \(\msf{PSL}_n(\mathbb{F}_q)\)'s), as well as, \(A_n\) and \(S_n\) for \(n \geq 5\), with 
                    \(A_5\) of order \(60\), being the smallest finite Oliver group. 
\end{itemize}
\smallskip

It is well-known that for any integer \(n \geq 1\), the real projective space \(\mbb{R}{\rm P}^{2n}\) has the fixed point property,
while each \(\mbb{R}{\rm P}^{2n-1}\)\! admits self-maps without fixed points. As \(\chi(\mbb{R}{\rm P}^{2n}) = 1\), 
it follows from Theorem~\ref{thm: F empty} that for any \(n \geq 1\), a finite group \(G\) has a smooth fixed point
free action on a compact smooth manifold \(M \simeq_{/G} \mbb{R}{\rm P}^{2n}\) if and only if \(n_G = 1\), i.e., 
\(G\) is an Oliver group. We show that in fact, each finite Oliver group \(G\) admits a smooth fixed point free action also on 
\(\mbb{R}{\rm P}^{2n}\) itself for some integer \(n \geq 1\).

\begin{theorem}\label{thm: F empty in RP2n}
Each finite Oliver group \(G\) has a smooth fixed point free action on some even dimensional real projective space.
\end{theorem}

In the books of Bredon \cite{Bredon 1972}, tom Dieck \cite{tom Dieck 1987}, and Kawakubo \cite{Kawakubo 1991},
one can find the background material on transformation groups. We refer the reader to \cite{Pawalowski 2002} 
for a survey about the fixed point sets of group actions on manifolds.

\section{Equivariant thickening}
\label{sec: thickening}

In this paper, without mentioning it explicitly, each smooth manifold \(M\) is \emph{second countable} (i.e., \(M\) has
 a countable base of topology).

For a finite group \(G\), each \(G\)-CW complex \(X\) that we consider is \emph{countable}, i.e., \(X\) is obtained from the disjoint union of countably 
many \(0\)-dimensional \(G\)-cells \(G/H\), by attaching countably many \(n\)-dimensional \(G\)-cells \(G/H \times D^n\) for various subgroups 
\(H\) of \(G\) and variuos integers \(n \geq 1\). Recall that \(X\) is said to be \emph{finite} (resp., \emph{finite dimensional}) 
if \(X\) has finitely many \(G\)-cells (resp., there exists an integer \(m\) such that \(m \geq n\) for all \(G\)-cells \(G/H \times D^n\) in \(X\)).

By the results of \cite{Illman 1983} or \cite{Matumoto-Shiota 1986}, any smooth \(G\)-manifold \(M\) has the structure of a finite dimensional, 
countable \(G\)-CW complex. Moreover, the manifold \(M\) is compact if and only if \(M\) has the structure of a finite \(G\)-CW complex.

We denote by the same symbol a real \(G\)-vector bundle and the total space of the bundle. For a real \(G\)-vector bundle \(E\) over \(X\) and 
a G-map \(f \colon M \to X\), \(f^{\ast} E\) denotes the real \(G\)-vector bundle over \(M\) induced by \(f\). Moreover, 
\(\mbb{R}^n_{\times X}\) denotes the product \(G\)-vector bundle \(\mbb{R}^n \times X\) over \(X\) with the trivial \(G\)-action on \(\mbb{R}^n\).
More generally, for a given \(\mathbb{R}G\)-module \(V\) (i.e., a real vector space \(V\) with a linear action of \(G\)), 
we denote by \(V_{\times X}\) the product \(G\)-vector bundle \(V \times X\) over \(X\).
If \(E\) and \(E'\) are two real vector bundles (resp., real \(G\)-vector bundles), we write \(E \approx E'\) (resp., \(E \approx_G E'\)) 
when \(E\) and \(E'\) are isomorphic as real vector bundles (resp., real \(G\)-vector bundles).

Now, we state two equivariant thickening theorems which allow to convert a \(G\)-CW complexes \(X\) into a smooth \(G\)-manifold \(M\)
of the \(G\)-homotopy type of \(X\), such that \(M^G = X^G\). Both theorems go back to 
\cite[Sections 2 and 3]{Pawalowski 1989}, where more detailed conclusions are given for any compact Lie group \(G\).

\begin{theorem}\label{thm: thickening F non-empty}
Let \(G\) be a finite group, let \(F\) be a non-empty smooth manifold, let \(X\) be a finite dimensional countable \(G\)-CW complex 
such that \(X^G = F\), and let \(E\) be a real \(G\)-vector bundle over \(X\) such that 
\[
(E|_F)^G \approx T F \oplus  \mbb{R}^n_{\times F} \quad \text{with} \quad n \geq 0.
\]
Then there exists a smooth \(G\)-manifold \(M\) of the \(G\)-homotopy type of \(X\), such that \(M^G\) is diffeomorphic to \(F\).
Moreover, there exists a strong \(G\)-deformation retraction \(f \colon M \to X\) such that for some \(\mathbb{R}G\)-module \(V\) with \(V^G = 0\),
\[
f ^{\ast}E \oplus V_{\times M} \approx_G T M \oplus \mbb{R}^n_{\times M}.
\]
\end{theorem}

\begin{theorem}\label{thm: thickening F empty}
Let \(G\) be a finite group, let \(X\) be a finite dimensional countable \(G\)-CW complex such that \(X^G = \varnothing\), Then, 
for any real \(G\)-vector bundle \(E\) over \(X\), there exists a smooth \(G\)-manifold \(M\) of the \(G\)-homotoy type of \(X\),
such that  \(M^G = \varnothing\), and there exists a strong \(G\)-deformation retraction \(f \colon M \to X\) such that 
for some \(\mathbb{R}G\)-module \(V\) with \(V^G = 0\), 
\[
f^{\ast}E \oplus V_{\times M} \approx_G TM.
\]
\end{theorem}

\begin{remark}
\emph{In  Theorems~\ref{thm: thickening F non-empty} and \ref{thm: thickening F empty},  if \(X\) is finite (resp.,  infinite),
\(M\) can be chosen to be compact (resp., open -- assuming \(\partial F = \varnothing\)). If in addition \(X\) is contractible, 
\(M\) can be chosen to be a disk (resp., Euclidean space).}
\end{remark}

\section{The tangent bundle restrictions}
\label{sec: tangent bundle}

A real vector bundle \(E\) over a space \(X\) is called \emph{stably complex} if for some integer \(n \geq 0\), the Whitney sum 
\(E \oplus \mbb{R}^n_{\times X}\) admits a complex structure.

Following \cite{Edmonds-Lee 1975}, a smooth manifold \(M\) is called \emph{stably complex} if the tangent bundle \(TM\) to \(M\) 
is stably complex, or equivalently, if there exists a smooth embedding of \(M\) into some Euclidean space such that the normal bundle 
of the embedding has a complex structure. It follows that the connected components of a stably complex manifold \(M\) all are either 
odd or even dimensional.

A smooth manifold \(M\) is called \emph{parallelizable} (resp., \emph{stably parallelizable}) if the connected components of \(M\) 
all have the same dimension, say \(m\), and
\[
T M \oplus \mbb{R}^n_{\times M} \approx \mbb{R}^{m+n}_{\times M}
\]
for \(n = 0\) (resp., \(n \geq 0\)). As noted in Section~\ref{sec: thickening}, \(\approx\) means that the two bundles in question 
are isomorphic as real vector bundles.

According to \cite[the proof of (3.2)]{Edmonds-Lee 1975}, any \(\mbb{F}_2\)-acyclic smooth manifold \(M\) is stably complex. However, 
for a prime \(p \geq 3\), an \(\mbb{F}_p\)-acyclic smooth manifold \(M\) is not necessarily stably complex.

By the Smith Theory and the work of Edmonds and Lee \cite[(3.1) and (3.2)]{Edmonds-Lee 1975}, the following proposition holds.

\begin{proposition}
Let \(G\) be a finite \(p\)-group for a prime \(p\), and let \(M\) be an \(\mbb{F}_p\)-acyclic stably complex \(G\)-manifold. Then
the fixed point set \(M^G\) is both \(\mbb{F}_p\)-acyclic and stably complex.
\end{proposition}

Now, we recall restrictions on the tangent bundles of the fixed point sets in the case of smooth actions of finite groups not of prime power order,
which are described by Oliver \cite{Oliver 1996}.

\smallskip\noindent
{\bf Group theory definitions.} Let \(\frak{P}\) (resp., \(\neg\frak{P}\)) be the class of finite groups of prime power order (resp., not of prime power order). 
We split  \(\neg \frak{P}\) into four mutually disjoint classes \(\frak{A}\), \(\frak{B}\), \(\frak{C}\), and \(\frak{D}\), defined as follows.

\begin{itemize}
\smallskip
\item[1)] \(G \in \frak{D}\) if \(G\) has a series of subgroups \(K \trianglelefteq H \leq G\) such that the quotient \(H/K\) is isomorphic to \(D_{pq}\), 
the dihedral group of order \(2pq\), for some two distinct primes \(p\) and \(q\) (and thus, \(G\) has an element not of prime power order, which is
conjugate to its inverse in \(G\)).
\smallskip
\item[2)] \(G \in \frak{C}\) if \(G \not\in \frak{D}\) and \(G\) has an element not of prime power order, which is conjugate to its inverse in \(G\).
\smallskip
\item[3)] \(G \in \frak{B}\) if \(G \not\in \frak{C} \cup \frak{D}\) and \(G\) has an element not of prime power order.
\smallskip
\item[4)] \(G \in \frak{A}\) if \(G \in \neg\frak{P}\) and each element of \(G\) has prime power order.
\end{itemize}

\bigskip\noindent
{\bf Representation theory definitions.} Let \(G\) be a finite group. For the field \(\mbb{F} = \mbb{C}\) or \(\mbb{R}\), two \(\mbb{F}G\)-modules 
\(V_0\) and \(V_1\) are \emph{primary matched} if \(V_0\) and \(V_1\) are ismorphic as \(\mbb{F}P\)-modules for any prime power order subgroup \(P\) of \(G\).
Now, we define three classes \(\frak{M}_{\mbb{C}}\),  \(\frak{M}_{\mbb{C}+}\), and \(\frak{M}_{\mbb{R}}\) of finite groups \(G\). 
\begin{itemize}
\smallskip
\item[1)] \(G \in \frak{M}_{\mbb{C}}\) if there exist two primary matched \(\mbb{C}G\)-modules \(V_0\) and \(V_1\) such that 
               \(\dim_{\mbb{C}} V_0^G = 0\) and \(\dim_{\mbb{C}} V_1^G = 1\).
\smallskip
\item[2)] \(G \in \frak{M}_{\mbb{C}+}\) if there exist two primary matched self-conjugate \(\mbb{C}G\)-modules \(V_0\) and \(V_1\) such that 
                \(\dim_{\mbb{C}} V_0^G = 0\) and \(\dim_{\mbb{C}} V_1^G = 1\).
\smallskip
\item[3)] \(G \in \frak{M}_{\mbb{R}}\) if there exist two primary matched \(\mbb{R}G\)-modules \(V_0\) and \(V_1\) such that
                \(\dim_{\mbb{R}} V_0^G = 0\) and \(\dim_{\mbb{R}} V_1^G = 1\).
\end{itemize}

\smallskip\noindent
The complexification \(V \otimes_{\mbb{R}}\mbb{C}\) of an \(\mbb{R}G\)-module \(V\) is a self-conjugate \(\mbb{C}G\)-module and clearly, 
if \(G \in  \frak{M}_{\mbb{C}}\), \(G\) is not of prime power order. Therefore, 
\[
\frak{M}_{\mbb{R}} \subset \frak{M}_{\mbb{C}+} \subset \frak{M}_{\mbb{C}} \subset \neg \frak{P}.
\]

The following lemma goes back to Oliver \cite[Lemma~3.1]{Oliver 1996}.

\begin{lemma}\label{lem: groups via modules}
For a finite group \(G\), the following three conclusions hold.
\begin{itemize}
\item[\rm 1)] \(G \in \frak{M}_{\mbb{C}}\) if and only if \(G\) has an element not of prime powoer order.
\item[\rm 2)] \(G \in \frak{M}_{\mbb{C}+}\) if and only if \(G\) has an element not of prime powoer order, which is conjugate to its inverse.
\item[\rm 3)] \(G \in \frak{M}_{\mbb{R}}\) if and only if G has a series of subgroups \(K \trianglelefteq H \leq G\) such that the quotient \(H/K\) 
                      is isomorphic to \(D_{pq}\), the dihedral group of order \(2pq\), for some two distinct primes \(p\) and \(q\).
\end{itemize}
\end{lemma}

By Lemma~\ref{lem: groups via modules} and the group theory definitions above,
\[
\frak{A} = \neg\frak{P} \setminus \frak{M}_{\mbb{C}}, \ \ \frak{B} = \frak{M}_{\mbb{C}} \setminus \frak{M}_{\mbb{C}+}, \ \
\frak{C} = \frak{M}_{\mbb{C}+} \setminus \frak{M}_{\mbb{R}}, \ \ \frak{D} = \frak{M}_{\mbb{R}}.
\]

For a  class \(\frak{G}\) of finite groups, let \(\frak{G}^{\triangleleft}\) denote the class of groups \(G \in \frak{G}\) with a normal \(2\)-Sylow subgroup. 
If \(G \in \frak{C}\) or \(\frak{D}\), then \(G\) does not contain a normal \(2\)-Sylow subgroup, i.e., 
\(\frak{C}^{\triangleleft} = \frak{D}^{\triangleleft} = \varnothing\).

Now, we split \(\neg\frak{P}\) into the following six mutually disjoint classes:
\[
\frak{A}^{\triangleleft}, \ \ \frak{B}^{\triangleleft}, \ \ \frak{A} \setminus \frak{A}^{\triangleleft}, \ \ \frak{B} \setminus \frak{B}^{\triangleleft}, 
\ \ \frak{C}, \ \ \frak{D}.
\]

\noindent
{\bf \msl{K}-theory definitions.} For a finite dimensional CW complex \(X\), let \(\widetilde{KO}(X)\), \(\widetilde{KU}(X)\), and \(\widetilde{KSp}(X)\) 
denote (respectively) the real, complex, and quaternion reduced \(K\)-theory of \(X\). Following Oliver \cite{Oliver 1996}, consider:

\begin{itemize}
\smallskip
\item[1)] the complexification of real vector bundles \(E\) over \(X\):\\ 
               \(c_{\mbb{R}} \colon \widetilde{KO}(X) \to \widetilde{KU}(X), [E] \mapsto [E \otimes_{\mbb{R}} (\mbb{C}_{\times X})]\)
\smallskip
\item[2)] the quaternionization of complex vector bundles \(E\) over \(X\):\\
                  \(q_{\mbb{c}} \colon \widetilde{KU}(X) \to \widetilde{KSp}(X), [E] \mapsto [E \otimes_{\mbb{C}} (\mbb{H}_{\times X})]\)
\smallskip
\item[3)] the complexification of quaternionic vector bundles \(E\) over \(X\):\\
                \(c_{\mbb{H}} \colon \widetilde{KO}(X) \to \widetilde{KU}(X), [E] \mapsto [\msf{Res}^{\mbb{H}}_{\mbb{C}}(E)]\)
\smallskip
\item[4)] the realification of complex vector bundles \(E\) over \(X\):\\
                \(r_{\mbb{C}} \colon \widetilde{KO}(X) \to \widetilde{KU}(X), [E] \mapsto [\msf{Res}^{\mbb{C}}_{\mbb{R}}(E)]\)
\end{itemize}

\medskip\noindent
{\bf The notion of  \(\widetilde{KO}(X, G)\).}
In the case where \(G \in \frak{P}\), set
\[
\widetilde{KO}(X, G) = r_{\mbb{C}}\big(\widetilde{KU}(X)\big).
\]

If \(F\) is a smooth manifold, then \([TF] \in \widetilde{KO}(F, G)\) if and only if \(F\) is stably complex. 
Therefore, in Theorem~\ref{thm: Y p-acyclic}, the restriction that \(F\) is stably complex can be restated as follows: \([TF] \in \widetilde{KO}(F, G)\).

\smallskip

In the case where \(G \in \neg\frak{P}\) and \(X\) is a finite CW complex, set
\begin{itemize}
\smallskip
\item[1)]  \(\widetilde{KO}(X, G) = r_{\mbb{C}}\big(\msf{tor}\,\widetilde{KU}(X)\big)\) for \(G \in \frak{A}^{\triangleleft}\).
\item[2)]  \(\widetilde{KO}(X, G) = r_{\mbb{C}}\big(\widetilde{KU}(X)\big)\) for \(G \in \frak{B}^{\triangleleft}\).
\item[3)]  \(\widetilde{KO}(X, G) = \msf{tor}\,\widetilde{KO}(X)\) for \(G \in \frak{A} \setminus \frak{A}^{\triangleleft}\).
\item[4)]  \(\widetilde{KO}(X, G) = \msf{tor}\,\widetilde{KO}(X) + r_{\mbb{C}}\big(\widetilde{KU}(X)\big)\) for \(G \in \frak{B} \setminus \frak{B}^{\triangleleft}\).
\item[5)]  \(\widetilde{KO}(X, G) = c_{\mbb{R}}^{-1}\big(\msf{tor}\, \widetilde{KU}(X) +  c_{\mbb{H}} \big(\widetilde{KSp}(X)\big)\big)\) for \(G \in \frak{C}\).
\item[6)]  \(\widetilde{KO}(X, G) = \widetilde{KO}(X)\) for \(G \in \frak{D}\).
\end{itemize}

\smallskip

The subgroup of \emph{quasidivisible elements} of an abelian group \(A\), denoted by \(\msf{qdiv}\, A\), is the intersection of the kernels 
of all homomorphisms from \(A\) into free abelian groups. If \(A\) is finitely generated, \(\msf{qdiv}\, A = \msf{tor}\, A\).

In the case where \(X\) is not finite, in the definition of   \(\widetilde{KO}(X, G)\), we replace the torsion groups 
\(\msf{tor}\,\widetilde{KO}(X)\) and \(\msf{tor}\,\widetilde{KU}(X)\) by the groups
\[
\msf{qdiv}\,\widetilde{KO}(X) \quad \text{and} \quad \msf{qdiv}\,\widetilde{KU}(X),
\]
respectively (cf., \cite[Theorems~0.1 and 0.2]{Oliver 1996}).

\begin{definition}\label{def: Oliver obstruction}
\emph{Let \(G\) be a finite group and let \(F\) be a smooth manifold with the trivial action of \(G\).
For a real \(G\)-vector bundle \(E\) over \(F\), the element
\[
\msf{Oliv}(E) \in \widetilde{KO}(F) \oplus \bigoplus_{P \leq G} \widetilde{KO}_P(F)_{(p)} / \msf{div}_{(p)}^{\infty}(F),
\]
called the \emph{Oliver obstruction} of \(E\), is defined by restricting the \(G\)-action on \(E\) to the trivial subgroup \(I\) of \(G\)
and to each \(p\)-subgroup \(P\) of \(G\) with \(p\big| |G|\), to obtain the following elements of the occurring summands.
\smallskip
\begin{itemize}
\item[1)] \([\msf{Res}^G_I(E)] \in \widetilde{KO}(F)\) and 
\item[2)] \([\msf{Res}^G_P(E)] +  \msf{div}_{(p)}^{\infty}(F) \in \widetilde{KO}_P(F)_{(p)} / \msf{div}_{(p)}^{\infty}(F)\), 
\end{itemize}
where \(\widetilde{KO}(F)_{(p)} / \msf{div}_{(p)}^{\infty}(F)\) is the quotient of the localized group \(\widetilde{KO}_P(F)_{(p)}\) 
by the subgroup \(\msf{div}_{(p)}^{\infty}(F)\) of infinitely \(p\)-divisible elements in \(\widetilde{KO}_P(F)_{(p)}\). 
}
\end{definition}

By \cite[Lemmas 3.1, 3.2, and 3.3]{Oliver 1996}, the condition that \(\msf{Oliv}(E) = 0\) for a real \(G\)-vector bundle 
\(E\) over \(F\) with \(E^G \approx TF\), is equivalent to the condition that \([TF] \in \widetilde{KO}(F, G)\) (cf., \cite[Theorem 0.2]{Oliver 1996}).

\begin{lemma}\label{lem: Oliver obstruction}{\rm (Oliver \cite{Oliver 1996})}
Let \(G\) be a finite group not of prime power order. Then, for a smooth manifold \(F\), there exists a real \(G\)-vector bundle \(E\) over \(F\)
such that \(E^G \approx T F\) and \(\msf{Oliv}(E) = 0\), if and only if \([TF] \in \widetilde{KO}(F, G)\).
\end{lemma}

Let \(G\) be a finite group and let \(M\) be a smooth \(G\)-manifold. Set \(F = M^G\) and \(E = TM|_F\), the \(G\)-vector bundle \(TM\) restricted to \(F\) . 
Then
\[
E \approx_G TF \oplus N,
\]
the Whitney sum of the tangent bundle \(TF\) upon which \(G\) acts trivially, and the \(G\)-equivariant normal bundle \(N\) of \(F\) in \(M\). 
For any point \(x \in F\), the fiber \(N_x\) over \(x\) of \(N\) is an \(\mbb{R}G\)-module with \(N_x^G = 0\) and thus, \(E^G \approx TF\).

By the arguments in \cite[the discussion after Theorem~0.1]{Oliver 1996}, the following proposition holds.

\begin{proposition}\label{pro: Oliver obstruction}{\rm (Oliver \cite{Oliver 1996})}
Let \(G\) be a finite group not of prime power order, and let \(M\) be a \(\mbb{Z}\)-acyclic smooth \(G\)-manifold with fixed point set \(F \neq \varnothing\).
Then
\[
E^G \approx TF \quad \text{and} \quad \msf{Oliv}(E) = 0,
\] 
where \(E = TM|_F\).
\end{proposition}

\section{Oliver's fixed point set theorems}
\label{sec: fixed point set}

First, we summarize the main results of Oliver \cite{Oliver 1975}, where \(F = X^G\) means that \(F\) and \(X^G\) are homeomorphic 
and \(F \simeq D^G\) means that \(F\) and \(D^G\) have the same homotopy type.

\begin{theorem}\label{thm: Oliver F non-empty}{\rm (Oliver \cite{Oliver 1975})}
Let \(G\) be a finite group not of prime power order, and let \(F\) be a finite CW complex such that \(F \neq \varnothing\). 
Then the following three conclusions are equivalent.
\begin{itemize}
\item[\rm 1)]  \(\chi(F) \equiv 1 \pmod{n_G}\).
\item[\rm 2)]  \(F = X^G\) for a finite contractible \(G\)-CW complex \(X\).
\item[\rm 3)]  \(F \simeq D^G\) for a smooth action of \(G\) on a disk \(D\). 
\end{itemize}
\end{theorem}

The following result justifies the notion of Oliver group.

\begin{theorem}\label{thm: fixed point free}{\rm (Oliver \cite{Oliver 1975})} 
Let \(G\) be a finite group. Then the following three conclusions are equivalent.
\begin{itemize}
\item[\rm 1)] \(G\) is not of prime power order and the Oliver number \(n_G = 1\).
\item[\rm 2)] There exists a finite contractible \(G\)-CW complex such that \(X^G = \varnothing\).
\item[\rm 3)] There exists a smooth action of \(G\) on a disk \(D\) such that \(D^G = \varnothing\).
\end{itemize}
\end{theorem}

The construction of \(G\)-CW complexes and real \(G\)-vector bundles described by Oliver \cite[the proof of Theorem 2.4]{Oliver 1996} 
yields the following theorem.

\begin{theorem}\label{thm: bundle extension}{\rm (Oliver \cite{Oliver 1996})}
Let \(G\) be a finite group not of prime power order, and let \(F\) be a smooth manifold. Then a real \(G\)-vector bundle \(E\) over \(F\) extends
to a real \(G\)-vector bundle over a finite $($resp., finite dimensional, countable$)$ contractible \(G\)-CW complex \(X\) such that \(X^G = F\), 
if and only if \(F\) is compact, \(\chi(F) \equiv 1 \pmod{n_G}\), and \(\msf{Oliv}(E) = 0\) $($resp., \(\msf{Oliv}(E) = 0\)$)$.
\end{theorem}

The results of \cite[Theorems 0.1, 0.2, and Lemmas 3.1, 3.2, 3.3]{Oliver 1996} allow to state the following theorem, 
where the manifold \(M\) can be chosen to be a disk (resp., Euclidean space) (cf., \cite[Thm. 2.4 and pp. 597–599]{Oliver 1996}).

\begin{theorem}\label{thm: fixed point set}{\rm (Oliver \cite{Oliver 1996})}
Let \(G\) be a finite group not of prime power order. Then a smooth manifold \(F\) is diffeomorphic to the fixed point set of a compact
$($resp., open$)$ contractible smooth \(G\)-manifold \(M\), if and only if \(F\) is compact, \(\chi(F) \equiv 1 \pmod{n_G}\), and 
\([TF] \in \widetilde{KO}(F,G)\), $($resp., the boundary \(\partial F = \varnothing\) and \([TF] \in \widetilde{KO}(F,G)\)$)$. 
\end{theorem}

Theorem~\ref{thm: fixed point set} follows from Lemma~\ref{lem: Oliver obstruction}, Proposition~\ref{pro: Oliver obstruction}, 
and Theorem~\ref{thm: bundle extension}, as well as, by using Theorem~\ref{thm: thickening F non-empty} to prove the sufficiency conclusion.

\section{Proofs of Theorems~\ref{thm: Y acyclic} and \ref{thm: F stably parallelizable}}
\label{sec: proofs of 0.2 and 0.3}

In \cite[the end of Sec.~2]{Oliver 1996}, Oliver made a comment about his construction of \(G\)-vector bundles presented in
\cite[Thm.~2.4]{Oliver 1996}, which reads as follows.

\emph{Theorem~\msl{2.4} can easily be combined with equivariant thickening, 
to allow the construction of smooth \(G\)-manifolds with various properties. But since it seems quite difficult to formulate 
such a theorem in the greatest possible generality, we limit the applications to the case of actions on disks and Euclidean spaces.}

It is possible to follows the ideas of  \cite[Thm.~2.4]{Oliver 1996} to prove Theorem~\ref{thm: sufficiency in thm Y acyclic}
below, but in order to simplify the arguments in the proof, we decided to apply recent results of Cappell, Weinberger, and Yan 
(see \cite{Cappell-Weinberger-Yan I 2022} and \cite{Cappell-Weinberger-Yan II 2022}). 
Here, we state only the result obtained in \cite[Thm.~2]{Cappell-Weinberger-Yan II 2022}, where \(G\) is a finite group not of prime power order, 
and \(Y\) is a \(G\)-template (see Sec.~0 for the notion of \(G\)-template). For \(Y = \msl{pt}\), the result below goes back to Oliver \cite{Oliver 1975}.


\begin{theorem}\label{thm: CWY 2020 F non-empty}
{\rm (Cappell-Weinberger-Yan \cite{Cappell-Weinberger-Yan II 2022})}  
Let \(G\) be a finite group not of prime power order, and let \(Y\) be a \(G\)-template. 
Then a finite CW complex \(F\) is homeomorphic to the fixed point set 
\(X^G\)\! of a finite \(G\)-CW complex \(X \simeq_{/G} Y\), if and only if
\[
\chi(F) \equiv \chi(Y^G) \pmod{n_G}.
\]
In particular, there exists a finite \(G\)-CW complex \(X \simeq_{/G} Y\) such that \(X^G = \msl{pt}\),
if and only if \(\chi(Y^G) \equiv 1 \pmod{n_G}\).
\end{theorem}

Now, we show a result which allows us to prove the sufficiency conlusions in Theorems~\ref{thm: Y acyclic} and \ref{thm: F stably parallelizable}.

\begin{theorem}\label{thm: sufficiency in thm Y acyclic}
Let \(G\) be a finite group not of prime power order, and let \(Y\) be a \(G\)-template such that \(\chi(Y^G) \equiv 1 \pmod{n_G}\).
Then any smooth manifold \(F\) with \([T F ] \in \widetilde{KO}(F,G)\)  is diffeomorphic to the fixed point set \(M^G\)
of a smooth \(G\)-manifold \(M \simeq_{/G} Y\). Moreover, if \(F\) is compact and \(\chi(F ) \equiv 1 \pmod{n_G}\), 
then \(M\) can be chosen to be compact.
\end{theorem}

\begin{proof}
As \([T F ] \in \widetilde{KO}(F,G)\),  Lemma~\ref{lem: Oliver obstruction} shows that 
there exists a real \(G\)-vector bundle \(E\) over \(F\) such that \(E^G \approx T F\) and \(\msf{Oliv}(E) = 0\). Hence, by Theorem~\ref{thm: bundle extension},
\(E\) extends to a real \(G\)-vector bundle \(E_0\) over \(X_0\), a finite dimensional, countable, contractible \(G\)-CW complex such that \(X_0^G = F\). Clearly,
\[
(E_0|_F)^G \approx E^G \approx TF.
\]
If \(F\) is compact and \(\chi(F) \equiv 1 \pmod{n_G}\), Theorem~\ref{thm: bundle extension} allows to assume that \(X_0\) is finite.
As \(\chi(Y^G) \equiv 1 \pmod{n_G}\) by the assumption, Theorem 4.1 shows that there exists a finite \(G\)-CW complex \(X_1 \simeq_{/G} Y\) 
such that \(X_1^G = \{x_1\}\). Choose a point \(x_0 \in X_0^G = F\) and form the wedge \(X = X_0 \vee X_1\) by identifying the two points \(x_0\) 
and \(x_1\). Then \(X^G = F\) and 
\[
X = X_0 \vee X_1 \simeq_{/G} (X_0 \vee X_1) / X_0 = X_1 \simeq_{/G} Y.
\]

Let \(f \colon X \to X_0\) be the \(G\)-map such that  \(f|_{X_0} = \msf{id}_{X_0}\)  and \(f(X_1) = \{x_0\}\).  
Let \(f^{\ast} E_0\) be the real \(G\)-vector bundle over \(X\) induced by \(f\). Then 
\[
(f^{\ast}E_0|_F)^G \approx (E_0|_F)^G \approx TF.
\]
Therefore, the \(G\)-vector bundle \(f^{\ast} E_0\) over \(X\) allows us to apply Theorem~\ref{thm: thickening F non-empty} to convert \(X\)
into a smooth \(G\)-manifold \(M \simeq_{/G} Y\) such that \(M^G = F\), which can be chosen to be compact whenever \(F\)  is compact
and \(\chi(F) \equiv 1 \pmod{n_G}\), completing  the proof of Theorem~\ref{thm: sufficiency in thm Y acyclic}.
\end{proof}

Now, we are ready to prove Theorem~\ref{thm: Y acyclic}.

\begin{proof}
In Theorem~\ref{thm: Y acyclic}, the necessity of the condition that \([TF] \in \widetilde{KO}(F, G)\) follows from Lemma~\ref{lem: Oliver obstruction} and 
Proposition~\ref{pro: Oliver obstruction}. In the case where \(F\) is compact, the necessity of the condition that \(\chi(F) \equiv 1 \pmod{n_G}\)
 is a consequence of  Theorem~\ref{thm: CWY 2020 F non-empty}, because in Theorem~0.1, \(Y\) is \(\mbb{Z}\)-acyclic and thus, \(\chi(Y) = 1\).

In Theorem~\ref{thm: Y acyclic}, the sufficiency of the conditions imposed on \(F\) follows from Theorem~\ref{thm: sufficiency in thm Y acyclic},
where the assumption that \(Y\) is \(\mbb{Z}\)-acyclic is replaced by the more general condition that \(\chi(Y) \equiv 1 \pmod{n_G}\), 
proving Theorem~\ref{thm: Y acyclic}.
\end{proof}

Next, we wish to prove Theorem~\ref{thm: F stably parallelizable}.

\begin{proof}
In Theorem~\ref{thm: F stably parallelizable}, the necessity of the condition that 
\[
(\ast) \quad \quad \chi(F) \equiv \chi(Y^G) \pmod{n_G}
\]
follows immediately from Theorem~\ref{thm: CWY 2020 F non-empty}.

In order to prove the sufficiency, assume that \((\ast)\) holds for a finite group \(G\) not of prime power order,
a compact stably parallelizable manifold \(F\), and  a \(G\)-template \(Y\).
Then, by Theorem~\ref{thm: CWY 2020 F non-empty}, there exists a finite \(G\)-CW complex \(X \simeq_{/G} Y\) such that \(X^G = F\).  
As \(F\) is stably parallelizable,
\[
TF \oplus \mbb{R}^n_{\times F} \approx \mbb{R}^{m+n}_{\times F}
\]
for \(m = \dim F\) and some integer \(n \geq 0\). Set \(E = \mbb{R}^{m+n}_{\times X}\).  Then
\[
(E|_F)^G \approx TF \oplus \mbb{R}^n_{\times F}
\]
and thus, using the \(G\)-vector bundle \(E\) over \(X\), we can apply Theorem~\ref{thm: thickening F non-empty} to convert \(X\) into 
a compact smooth \(G\)-manifold \(M \simeq_{/G} Y\) such that \(M^G = F\), completing the proof of Theorem~\ref{thm: F stably parallelizable}.
\end{proof}

\section{Proofs of Theorem~\ref{thm: F empty}  and \ref{thm: F empty in RP2n}}
\label{sec: proofs of 0.4 and 0.5}

First, we recall the result of Cappell, Weinberger, and Yan \cite[Thm.~2 ]{Cappell-Weinberger-Yan II 2022}, 
in the case where \(F = \varnothing\) (cf., \cite[Thm.~6 and Sec.~4]{Cappell-Weinberger-Yan II 2022}).

\begin{theorem}\label{thm: CWY 2020 F empty}
{\rm (Cappell-Weinberger-Yan \cite{Cappell-Weinberger-Yan II 2022})}   
Let \(G\) be a finite group not of prime power order, and let \(Y\) be a \(G\)-template. Then there exists a finite \(G\)-CW complex \(X \simeq_{/G} Y\) 
such that \(X^G = \varnothing\), if and only if \(\chi(Y^G) \equiv 0 \pmod{n_G}\).
\end{theorem}

Now, we are ready to prove Theorem~\ref{thm: F empty}.

\begin{proof}
In Theorem~\ref{thm: F empty},  the necessity of the condition that 
\[
(\ast) \quad \quad \chi(Y^G) \equiv 0 \pmod{n_G}
\]
follows from Theorem~\ref{thm: CWY 2020 F empty}. In order to prove the sufficiency, note that once \((\ast)\) holds, 
it follows from Theorem~\ref{thm: CWY 2020 F empty} that there exists a finite \(G\)-CW complex \(X \simeq_{/G} Y\) such that \(X^G = \varnothing\). 
Thus, using any real \(G\)-vector bundle \(E\) over \(X\), we can apply Theorem~\ref{thm: thickening F empty} to convert \(X\) into a compact smooth \(G\)-manifold 
\(M \simeq_{/G} Y\) such that \(M^G = \varnothing\), completing the proof of Theorem~\ref{thm: F empty}.
\end{proof}

By the work of  Laitinen and Morimoto \cite{Laitinen-Morimoto 1998}, a finite group \(G\) has a smooth action on a sphere with exactly one fixed point,
 if and only if \(G\) is an Oliver group. The necessity of the conclusion that \(G\) is an Oliver group follows from the Slice Theorem and 
Theorem~\ref{thm: fixed point free}, while the sufficiency of the conclusion is the main result of \cite{Laitinen-Morimoto 1998} and it reads as follows.

\begin{theorem}\label{thm: one fixed point}
Each finite Oliver group \(G\) has a smooth action on a sphere \(S^{2n}\) with exactly one fixed point \(x_0\) at which the tangent 
\(\mbb{R}G\)-modul \(T_{x_0} (S^{2n})\) is without large isotropy subgroups outside of the origin.
\end{theorem}

We recall that for a finite group \(G\), a subgroup \(H\) of \(G\) is called \emph{large} in \(G\) if there exists a prime \(p\) 
dividing the order of \(G\), such that
\[
G^p \leq H \leq G,
\]
where \(G^p\) is the co-Sylow \(p\)-subgroup of \(G\), i.e., the smallest normal subgroup \(G^p\) of \(G\) such that the quotient group 
\(G/G^p\) is of \(p\)-power order.

Now, we are able to prove Theorem~\ref{thm: F empty in RP2n} which asserts that each finite Oliver group \(G\) has a smooth fixed point free action 
on some even dimensional real projective space.

\begin{proof}
Let \(G\) be a finite Oliver group. By Theorem~\ref{thm: one fixed point}, there exists a smooth action of \(G\) on some sphere \(S^{2n}\)\! 
with exactly one fixed point, say \(x_0\), at which the tangent \(\mbb{R}G\)-module \(T_{x_0}(S^{2n})\) is without large isotropy subgroups 
outside of the origin. Set \(V = T_{x_0}(S^{2n})\) and let \(DV\) be the \(G\)-invariant unit disk of \(V\). Consider the real projective space 
\(\mbb{R}{\rm P}^{2n}\)\! as the quotient space
\[
\mbb{R}{\rm P}^{2n} = DV/_{x \sim -x}
\]
obtained from \(DV\) by identifying the antipodal points \(x\) and \(-x\) for all \(x \in SV\), the \(G\)-invariant unit sphere of \(V\).
The construction yields a smooth action of \(G\) on \(\mbb{R}{\rm P}^{2n}\)\!. In other words, \(\mbb{R}{\rm P}^{2n}\)\! with the action 
of \(G\) is the projectivization \(P(V \oplus \mbb{R})\) of the \(\mbb{R}G\)-module \(V\). Clearly, \(V^G = \{0\}\).

For the origin \(0\) in \(DV\), the class \([0]\) in \(\mbb{R}{\rm P}^{2n}\)\! is a fixed point of the action of \(G\) on \(\mbb{R}{\rm P}^{2n}\)\!, 
and for any \(x \in SV\), \(G(x) \neq \{x, -x\}\) as otherwise, \(G_x\) would be a subgroup of \(G\) of index \(2\) and thus, \(G_x\) would be a large 
subgroup of \(G\). So, for \(x \in S(V)\), the class \([x]\) in \(\mbb{R}{\rm P}^{2n}\)\! is not a fixed point of the action.
Hence, \([0]\) is the single fixed point of the action of \(G\) on \(\mbb{R}{\rm P}^{2n}\)\!.

The tangent \(\mbb{R}G\)-modules \(T_{x_0}(S^{2n})\) and \(T_{[0]}(\mbb{R}{\rm P}^{2n})\) are isomorphic and thus,
the Slice Theorem allows us to form the \(G\)-equivariant connected sum
\[
S^{2n} \# \ P(V \oplus \mbb{R})  \cong \ \mbb{R}{\rm P}^{2n}\!,
\]
by removing \(G\)-invariant open disks around \(x_0 \in S^{2n}\)\! and \([0] \in  P(V \oplus \mbb{R})\)  and
glueing together their boudaries. As the construction deletes the single fixed points of the actions of \(G\) on \(S^{2n}\)\! and  \(P(V \oplus \mbb{R})\), 
we obtain a smooth fixed point free action of \(G\) on \(\mbb{R}{\rm P}^{2n}\)\!, completing the proof of Theorem~\ref{thm: F empty in RP2n}.
\end{proof}



\newcommand{\address}{Krzysztof  M. Pawa{\l}owski\\ 
                                         Adam Mickiewicz University in Pozna{\'n}\\ 
                                         ul. Uniwersytetu Pozna{\'n}skiego 4,  61-001 Pozna{\'n}, Poland\\
                                         \url{kpa@amu.edu.pl}}

\bigskip
\bigskip
\bigskip

\noindent
\author{\emph{Krzysztof M. Pawa{\l}owski}, e-mail: kpa@amu.edu.pl\\ Faculty of Mathematics and Computer Science\\
             Adam Mickiewicz University in Pozna{\'n}\\ ul. Uniwersytetu Pozna{\'n}skiego 4, 61-614 Pozna{\'n}, Poland}

\bigskip

\noindent
\author{\emph{Jan Pulikowski}, e-mail: pulik@amu.edu.pl\\ Faculty of Mathematics and Computer Science\\
             Adam Mickiewicz University in Pozna{\'n}\\ ul. Uniwersytetu Pozna{\'n}skiego 4, 61-614 Pozna{\'n}, Poland}

\end{document}